\documentclass[12pt]{article}
\usepackage{amsmath,amssymb,amsfonts,amsthm,indentfirst,setspace,graphics}
\usepackage{latexsym,bm}
\usepackage{mathrsfs,amsmath,amssymb,amsthm}
\usepackage[dvipdfm]{graphicx}

\theoremstyle{plain}
\newtheorem{theo}{Theorem}[section]%\newtheorem{theo}{Theorem}[section]

\newtheorem{lem}{Lemma}[section]
\newtheorem{cor}{Corollary}[section]
\newtheorem{rem}{Remark}[section]

\newcommand{\be}{\begin{equation}}
\newcommand{\ee}{\end{equation}}
\newcommand{\bea}{\begin{eqnarray}}
\newcommand{\eea}{\end{eqnarray}}
\newcommand{\eeas}{\end{eqnarray*}}
\newcommand{\beas}{\begin{eqnarray*}}
\usepackage{amscd}

\numberwithin{equation}{section}
\newcommand{\la}{\langle}
\newcommand{\ra}{\rangle}

\makeatletter
\newcommand{\trace}{\mathop{\operator@font Trace}}
\newcommand{\vspan}{\mathop{\operator@font Span}}
\newcommand{\Int}{\mathop{\operator@font Int}}
\newcommand{\grad}{\mathop{\operator@font grad}}
\newcommand{\diver}{\mathop{\operator@font div}}
\newcommand{\id}{\mathop{\operator@font id}}
\newcommand{\Ad}{\mathop{\operator@font Ad}}

\makeatother

\begin{document}

\title{Generalized Sasakian space forms and Riemannian manifolds of quasi constant sectional curvature
%\thanks{This work was supported in part by the UMRG research grant (Grant No. RG163/11AFR)}
}

\author{Avik De and Tee-How \textsc{Loo}\\
Faculty of Engineering and Science, Universiti Tunku Abdul Rahman\\
50744 Kuala Lumpur, Malaysia.\\ 
Institute of Mathematical Sciences, University of Malaya \\
50603 Kuala Lumpur, Malaysia.	\\ 
\ttfamily{de.math@gmail.com},
\ttfamily{looth@um.edu.my}}

\date{}
\maketitle

\abstract{In this paper, we show that a generalized Sasakian space form of dimension greater than three  is either of  constant sectional curvature; or  a canal hypersurface in Euclidean or Minkowski spaces; or locally a certain type of twisted product of  a real line  and a flat almost Hermitian manifold; or locally a wapred product of a real line and a generalized complex space form; 
or an $\alpha$-Sasakian space form; or it is of five dimension and admits an $\alpha$-Sasakian Einstein structure. 
In particular, a local classification for generalized Sasakian space forms of dimension greater than five is obtained.
A local classification of Riemannian manifolds of quasi constant sectional curvature of dimension greater than three is also given in this paper.}


 Riemannian manifolds of quasi constant sectional curvature.  Trans-Sasakian manifolds.  

%%%%%%%%%%%%%%%%%%%%%%%%%%%%%%%%%%%%%%%%%%%%%%%%%%%%%%%%%%%%%%%%%%%%%%%%%%%%%%%%%%%%%%%%%%%%%%%%%%%%%%%%%%%%

\section{Introduction}
A \emph{generalized complex space form} is a  $RK$-manifold with pointwise constant holomorphic sectional curvature and of constant type.
It is known that an almost Hermitian manifold $P(F_1,F_2)$ with almost complex structure $J$ is a  
generalized complex space form if  
its Riemannian curvature tensor $R^P$ satisfies
\begin{align*}
R^P(X,Y)Z=
&F_1\{\la Y,Z\ra X-\la X,Z\ra Y\}  \\
&+F_2\{\la JY,Z\ra JX-\la JX,Z\ra JY-2\la JX,Y\ra JZ\}	
\end{align*}
where $F_1$ and $F_2$ are functions on $P$ (cf. \cite{tricerri-vanhecke, vanhecke}). 

An important characterization of such spaces of higher dimension was given in \cite{tricerri-vanhecke}.

\begin{theo}[\cite{tricerri-vanhecke}] \label{thm:gcsf}
Let $P(F_1,F_2)$ be a generalized complex space form with  $\dim_{\mathbb C}P\geq 3$. If $F_2$ is not identically zero, then 
$P$ is a non-flat complex space form (that is, an open part of a complex projective space or a complex hyperbolic space).
\end{theo}
We remark that the above theorem is not true in general.
Indeed, a generalized complex space form $P$ with  $\dim_{\mathbb C}P=2$ and with nonconstant $F_2$ can always be constructed via certain conformal deformations on a Bochner flat Kaehler manifold with nonconstant scalar curvature  (cf. \cite{olszak}).

The almost contact counterpart, so-called generalized Sasakian space forms was introduced in \cite{abc}.
A \emph{generalized Sasakian space form} $M(f_1,f_2,f_3)$ is an almost contact metric manifold with almost contact metric structure 
$(\phi, \xi,\eta,\la,\ra)$ whereby its curvature tensor satisfies 
\begin{align}\label{eqn:gs}
R=&f_1R_1+f_2R_2+f_3R_3	
\end{align}
for some functions $f_1$, $f_2$ and $f_3$ on $M$, where $R_1$, $R_2$ and $R_3$ are curvature-like tensors given by 
\begin{align*}%\label{eqn:gs}
R_1(X,Y,Z)=&\la Y,Z\ra X-\la X,Z\ra Y  \nonumber\\
R_2(X,Y,Z)=&\la\phi Y,Z\ra\phi X-\la\phi X,Z\ra\phi Y-2\la\phi X,Y\ra\phi  Z	\nonumber \\
R_3(X,Y,Z)=&-\eta(Y)\eta(Z)X-\la Y,Z\ra\eta(X)\xi+\eta(X)\eta(Z)Y+\la X,Z\ra\eta(Y)\xi.	
\end{align*}
An algebraic characterization was obtained in \cite{falcitelli}: a generalized Sasakian space form $M$  is an almost contact $N(k)$-manifold with pointwise constant $\phi$-sectional curvature and 
\[
\la R(X,Y)X,Y\ra-\la R(X,Y)\phi X,\phi Y\ra=l\{||X||^2||Y||^2-\la X,Y\ra^2-\la X,\phi Y\ra^2\}
\]
for any $X$ and $Y$ orthogonal to $\xi$, where $l$ is a function on $M$.

Typical examples of generalized Sasakian space forms are Sasakian, cosymplectic or Kenmotsu space forms.
More precisely,  for a constant $c$,  a generalized Sasakian space form $M$ becomes
\begin{itemize}
	\item[(a)] a Sasakian space form when $M$ is Sasakian with $f_1 = (c+4)/4$,  and $f_2 = f_3 =(c-1)/4$; and 
	\item[(b)]  a cosymplectic space form when $M$ is cosymplectic with $f_1=f_2=f_3=c/4$; and  
	\item[(c)]  a Kenmotsu space forms when $M$ is Kenmotsu with $f_1=(c-3)/4$ and $f_2 = f_3 =(c+1)/4$. 
\end{itemize}

Apart from the spaces mentioned above, the class of generalized Sasakian space forms does include some other spaces, for  instance,
the five-dimensional sphere $S^5$ with a nearly Sasakian manifold structure and 
the warped products of $\mathbb R$ and a generalized complex space form $P$. 

Generalized Sasakian space forms have been studied in a number of papers from several points of view
(for instance, \cite{abc}--\cite{alegre-carriazo},  \cite{falcitelli},  \cite{ghosh}, \cite{kim},  etc).
To certain extent, we may say that these papers intended to provide new examples of generalized Sasakian space forms;  
to characterize subclasses of generalized Sasakian space forms with specific geometric properties
and to identify automorphisms preserving the generalized Sasakian space form structure.
Nevertheless, the complete classification of generalized Sasakian space forms has yet to be obtained up to this point.
In this paper, we shall provide a complete list of spaces of  dimension greater than three that possibly carry a generalized Sasakian spaces form structure. 
%More precisely, we 
%give a local classification of generalized Sasakian space forms of higher dimension. More precisely, we 
%prove the following result:  

The idea of getting the list was comparatively simple though a large number of lengthly calculations needed in the proof.
We shall first derive several important data out of the second Bianchi's identity, and then considered cases, which are divided  according to  the characteristics of the functions $f_1, f_2$ and $f_3$. 
%With all those obtained data, the classification will  then be obtained after investigation on each of the cases. 

Among the cases we considered,  generalized Sasakian space forms with $f_2=0$ seem to have their unique characteristics. 
Firstly, these spaces are  locally conformally flat (cf. \cite{kim}), secondly,  they are also of quasi constant sectional curvature.  
It is inevitable to include a study on Riemannian manifolds of quasi constant sectional curvatures  in our paper.
%in order to complete the classification. 

A Riemannian manifold $M$ is said to have \emph{quasi constant sectional curvature} if it is equipped with a unit vector field $\xi$ and a $1$-form 
$\eta$ given by $\eta(\cdot) =\la\xi,\cdot\ra$ such that its Riemannian curvature tensor $R$ satisfies
(cf. \cite{chen-yano})
\[
R=f_1R_1+f_3R_3
\]
for some functions $f_1$ and $f_3$ on $M$.

Geometrically, a Riemannian manifold $M$ equipped with a unit vector field $\xi$ is of quasi constant sectional curvature 
if for any plane $E$ in $T_pM$, $p\in M$, the sectional curvature of $E$ depends only on the point $p$ and the angle between $E$ and $\xi$.
%A generalized Sasakian space form with $f_2=0$ is of quasi constant sectional curvature.

The study of Riemannian manifolds of quasi constant sectional curvature was closely related to locally conformally flat hypersurfaces in Euclidean spaces. 
Indeed, in \cite{ganchev,ganchev2}, Ganchev and Mihova showed that a Riemannian manifold of quasi constant sectional curvature 
with $f_1f_3\neq0$ is an open part of a (space-like) canal hypersurface in Euclidean spaces or Minkowski spaces (cf. Theorem~\ref{thm:canal}).
%However, Ganchev and Mihova's list was incomplete as the case $f_1=0$ and $f_3\neq0$ was not included in their papers.
By including the case $f_1=0$ and $f_3\neq0$, we obtain
%fill in the gap of Ganchev and Mihova's list by adding in this case. 
\begin{theo}\label{thm:quasi}
A Riemannian manifold $M$ of dimension $m\geq 4$  is of quasi-constant sectional curvature 
if and only if  one of the following holds:
\begin{enumerate}
\item[(a)]  $M$ is of constant sectional curvature;
\item[(b)]  $M$ is an open part of a canal hypersurface in $\mathbb R^{m+1}$;  
\item[(c)]   $M$ is an open part of a space-like canal hypersurface in $\mathbb R^{m+1}_1$ of elliptic type;  
\item[(d)]   $M$ is an open part of a space-like canal hypersurface in $\mathbb R^{m+1}_1$ of hyperbolic type;  
\item[(e)]   $M$ is an open part of a space-like canal hypersurface in $\mathbb R^{m+1}_1$ of parabolic type;
\item[(f)]   $M$ is locally a twisted product space  ${_a}\mathbb R\times\mathbb R^m$ with 
		\[
	a(t,x)=\sum^m_{i=1}(C(t)x_i^2+D_i(t)x_i)+E(t); \quad f_3=\frac{2C}a
	\]
	for some functions  $C\neq0$, $D_1,\cdots, D_m$ and  $E$ on $\mathbb R$  such that $a>0$. 
	\end{enumerate}
\end{theo}

The above theorem provides all the possibilities for generalized Sasakian space forms in the case $f_2=0$.
Next, we shall show that generalized Sasakian space forms with $f_2\neq0$ must be either locally a warped product of $\mathbb R$ and generalized complex space forms; or $\alpha$-Sasakian; or carry an $\alpha$ Sasakian Einstein structure. The precise statement is given as follows.
\begin{theo} \label{thm:gssf} %\label{theo:f_2=0_a}
Let $M^{2n+1}(f_1,f_2,f_3)$ be a generalized Sasakian space form, $n\geq 2$. Then  one of the following holds:
\begin{enumerate}
\item[(a)] $M$ is of constant sectional curvatures; or
\item[(b)]  $M$ is an open part of a canal hypersurface in $\mathbb R^{8}$, $n=3$;  or 
\item[(c)]  $M$ is an open part of a space-like canal hypersurface in $\mathbb R^{8}_1$ of elliptic type, $n=3$;  or
\item[(d)]  $M$ is an open part of a space-like canal hypersurface in $\mathbb R^{2n+2}_1$ of hyperbolic type;  or 
\item[(e)]  $M$ is an open part of a space-like canal hypersurface in $\mathbb R^{2n+2}_1$ of parabolic type; or
\item[(f)]
$M$ is locally a twisted product  $_a\mathbb R\times P$, where $P$ is a flat almost Hermitian manifold and 
		\[
	a(t,x)=\sum^{2n}_{i=1}(C(t)x_i^2+D_i(t)x_i)+E(t); \quad f_3=\frac{2C}a
	\]
	for some functions  $C\neq0$, $D_1,\cdots, D_{2n}$ and  $E$ on $\mathbb R$  such that $a>0$. 
\item[(g)]
$M$ is locally a warped product $\mathbb R\times_{b}P$ with $\ddot b=-(f_1-f_3)b $ and $\beta=\dot b/b$, where  $P(F_1,F_2)$ is a generalized complex space form with functions
$F_1=b^2f_1+\dot b^2$ and $F_2=b^2f_2$. 
In particular, for $n\geq3$, $M$ is a $\beta$-Kenmotsu manifold of pointwise constant $\phi$-sectional curvatures; or 

\item[(h)]  
$M$ is an $\alpha$-Sasakian space form  with $f_1-f_3=\alpha^2$; or
\item[(i)]
$M$ admits an $\alpha$-Sasakian Einstein structure $(\psi,\xi,\eta,\la,\ra)$ with $\psi\phi=\phi\psi$,  $\psi^2=\phi^2$ and $n=2$, where $\nabla\xi=-\alpha\psi$.
\end{enumerate}
\end{theo}

\begin{rem}
\begin{enumerate}
\item[(i)]  The above theorem gives a local classification for generalized Sasakian space forms of dimenion greater than five.
The authors do not know any example in (i) with $\psi\neq\phi$.

\item[(ii)]
The function $f_2=0$ for spaces in  (b)--(f).
\item[(iii)]
A generalizes Sasakian space form with $f_2\neq0$ %is either $\alpha$-Sasakian or $\beta$-Kenmotsu, and
belongs to (g)--(i).
\item[(iv)]
The classes of spaces in the above theorem are not mutually exclusive.
An  $\beta$-Kenmotsu manifold $M$  which is (locally)  a warped product of $\mathbb R$ and a complex Euclidean space  
belongs to both  (e) and (g);  a Sasakian odd-dimensional sphere belongs to both  (a) and (h).
Besides, spaces in (b)--(e) are also included in (g) if they are subprojective (see \cite[pp. 329]{schouten} for precise definition).
 
\end{enumerate}
\end{rem}

Throughout this paper, all manifolds are assumed to be smooth and  connected.
For an $m$-dimensional Riemannian manifold $M$, 
we denote by $X^\flat$ the $1$-form dual to a vector field $X$ on $M$, $\omega^\sharp$ the vector field associated to an $1$-form $\omega$ on $M$,
by $d=\sum^{m}_{j=1}E^\flat_j\wedge\nabla_{E_j}$  the differential  and  
$\delta=-\sum^{m}_{j=1}i_{ E_j} \nabla_{E_j}$ the codifferential operators on $M$, 
where $\{E_1,\cdots,E_{m}\}$ is a local orthonormal frame field on $M$. % and $E_j^\flat$ is the $1$-form dual to $E_j$.
Furthermore,  if there exists a unit vector field $\xi$ on $M$, 
we denote by $V=\nabla_{\xi}\xi$ and $\bar{X}=X-\eta(X)\xi$, for each $X\in TM$, where $\eta=\xi^\flat$.
As we consider local geometry of $M$, for a function $f$ on $M$, unless otherwise stated, by $f=0$ (resp. $f\neq0$) we mean $f$ is identically zero on $M$ 
(resp. $f$ is nowhere zero on $M$).
%, $V_2=\nabla_\xi V$, $B=V-\beta\xi$, where $\beta=(1/2n)\diver\xi$.
%Furthermore, we denote by $v$, $v_2$ and $b$ respectively the $1$-form dual to $V$, $V_2$ and $B$. 

%%%%%%%%%%%%%%%%%%%%%%%%%%%%%%%%%%%%%%%%%%%%%%%%%%%%%%%%%%%%%%%%%%%%%%
%%%%%%%%%%%%%%%%%%%%%%%%%%%%%%%%%%%%%%%%%%%%%%%%%%%%%%%%%%%%%%%%%%%%%%
\section{Almost contact metric manifolds} 

A $(2n+1)$-dimensional Riemannian manifold $M$ with Riemannian metric  $\la, \ra$ is said to be almost contact metric if there exist on $M$ 
a $(1,1)$-tensor field $\phi$, a vector field $\xi$ and a $1$-form $\eta$ such that 
\begin{align}%\label{eqn:contact}
\phi ^{2} = -\mathbb I_{TM} + \xi\otimes \eta,
	\quad \phi \xi& = 0,
	\quad \eta\circ \phi = 0,  \nonumber \\
	%\end{align*}
%\begin{align}
\label{eqn:contact-metric}
\la\phi\:\cdot,\phi\:\cdot\ra	=\la\:\cdot,\cdot\ra&-\eta\otimes\eta.
\end{align}
The fundamental $2$-form $\Phi$ is given by $\Phi(\cdot,\cdot)=\la\:\cdot,\phi\:\cdot\ra$ and the CR-distribution $D$ is defined by $D=\ker\eta$.
Denote by $\nabla$ the Levi-Civita connetion on $M$.
If there are two functions $\alpha$ and $\beta$ on $M$ such that 
\[
(\nabla_X\phi)Y=\alpha\{-\eta(Y)X+\la  X,Y\ra\xi\}+\beta\{-\eta(Y)\phi X+\la\phi X,Y\ra\xi\}
\]
for any $X$, $Y\in TM$, 
then $M$ is called an \emph{$(\alpha,\beta)$-trans-Sasakian manifold}.
In particular,  $(0,\beta)$-trans-Sasakian manifolds are called  \emph{$\beta$-Kenmotsu manifolds} while  $(\alpha,0)$-trans-Sasakian manifolds are called   $\alpha$-Sasakian manifolds. 
By Sasakian (resp. Kenmotsu) manifolds, we mean $1$-Sasakian (resp. $1$-Kenmotsu) manifolds.
Further, an $0$-Kenmotsu manifold is called a cosymplectic manifold.
It is known that an $(\alpha,\beta)$-trans-Sasakian manifold of dimension greater than $3$ is either $\alpha$-Sasakian or $\beta$-Kenmotsu 
(cf. \cite{marrero}).

Let  $(P,g_P)$ be an $m$-dimensional Riemannian manifold.
Consider a doubly twisted product manifold $M=_a\mathbb R\times_b P$ with the doubly twsited product metric
\[
\la\ ,\ \ra=a^2\pi^*_1dt^2+b^2\pi^*_2g_P 
\]
where $a$ and $b$ are positive functions on $\mathbb R\times P$, 
$t$ is  the standard coordinate of $\mathbb R$, and 
$\pi_1$ and $\pi_2$ are  the projections from $\mathbb R\times P$ on $\mathbb R$ and $P$ respectively. 

Throughout this paper, we shall denote by the same $T$ the lift of a tensor field $T$ of $P$ (or $\mathbb R$) for the sake of simplicity.

Now we consider an almost Hermitian manifold $(P,J,g_P)$   with fundamental $2$-form 
$\Omega(\cdot,\cdot)=g_P(\cdot, J\cdot)$, we can define an almost contact structure on $M=_a\mathbb R\times_b P$ by
\begin{align}\label{eqn:cont_warped-product}
\Phi=b^2\Omega, \quad \eta=adt, \quad \xi=\eta^\sharp.
\end{align}
Note that for any $X$, $Y$,  $Z\in TM$, we have  
\begin{align*}% \label{eqn:R_warped-product}
R(X  ,Y)Z
= & R(\bar X,\bar Y)\bar Z	+\eta(X)R(\xi,\bar Y)\bar Z-\eta(Y)R(\xi,\bar X)\bar Z +\eta(Z)R(\bar X,\bar Y)\xi\\
   & -\eta(Z)\eta(X)R(\bar Y,\xi)\xi+\eta(Z)\eta(Y)R(\bar X,\xi)\xi.
\end{align*}
In particular, for a warped product $M=\mathbb R\times_b P$, that is, $a=1$ and $b$ depends only on $\mathbb R$,
 by  applying \cite[Proposition 42]{oneill} in the above equation,
we can derive
\begin{align} \label{eqn:R_warped-product}
&R(X  ,Y)Z=R^P(\bar X,\bar Y)\bar Z		-\frac{\dot b^2}{b^2}\{\la Y,Z\ra X-\la X,Z\ra Y\}	\nonumber \\
%=&-\frac{\dot\rho^2}{\rho^2}\{\la\bar Y,Z\ra\bar X-\la\bar X,Z\ra\bar Y\}\nonumber\\
 % &+\frac{\ddot\rho}\rho\{\eta(Z)\eta(X)\bar Y-\eta(Z)\eta(Y)\bar X-\la\bar Y,Z\ra\eta(X)\xi+\la\bar X,Z\ra\eta(Y)\xi\}	\nonumber\\
&+	\left(\frac{\ddot b}b-\frac{\dot b^2}{b^2}	\right)
	   \{\eta(Z)\eta(X) Y
		-\eta(Z)\eta(Y) X-\la Y,Z\ra\eta(X)\xi+\la X,Z\ra\eta(Y)\xi\}
\end{align}
for any $X$,  $Y$, $Z\in TM$, where $R^P$ is the Riemannian curvature tensor of $P$ and $\dot b=\partial_t b$. 

The following results can be obtained directly from the above equation.

%\begin{theo} \label{theo:warped-product-Nk}
%Let  $P$ be an almost Hermitian manifold. Then, with the almost contact structure (\ref{eqn:cont_warped-product}), the warped product 
%$M=\mathbb R\times_{b} P$ is an $N(k)$-manifold with function $k=-\ddot b/b$.
%\end{theo}

\begin{theo}[\cite{abc}] \label{theo:warped-product-gSSF}
Let  $P(F_1,F_2)$ be a generalized complex space form. Then, with the almost contact structure (\ref{eqn:cont_warped-product}), the warped product $M=\mathbb R\times_{b} P$ is a generalized Sasakian space form with functions   
\begin{align*}
f_1=\frac{F_1-\dot b^2}{b^2}, \quad 
f_2=\frac{F_2}{b^2}, \quad 
f_3=\frac{F_1-\dot b^2}{b^2}+\frac{\ddot b}b.
\end{align*}
\end{theo}

%%%%%%%%%%%%%%%%%%%%%%%%%%%%%%%%%%%%%%%%%%%%%%%%%%%%%%%%%%%%%%%%%%%%%%
%%%%%%%%%%%%%%%%%%%%%%%%%%%%%%%%%%%%%%%%%%%%%%%%%%%%%%%%%%%%%%%%%%%%%%

\section{$N(k)$-manifolds} 
% $ \omega^\sharp=g(X^\flat,Y)\omega^\#$

A Riemannian manifold  $M$ is called an  $N(k)$-manifold if  there exists a unit vector field $\xi$ on $M$ such that 
the curvature tensor satisfies the $k$-nullity condition:
\begin{align}\label {la1.1}
R(X,Y)\xi=k\{\eta(Y)X-\eta(X)Y\}, \quad X,Y\in TM 
\end{align}
where $k$ is a function on $M$. 
In particular, 
a generalized Sasakian space form is an $N(k)$-manifold with $k=f_1-f_3$.

\begin{theo} \label{theo:Nk-Sasakian}
Let $M$ be an $N(k)$-manifold of dimension $m\geq 5$. Suppose  
\[
\la\nabla_X\xi,Y\ra+\la\nabla_Y\xi, X\ra=2\beta\la\bar X,Y\ra, \quad X,Y\in TM
\]
where $\beta$ is a function on $M$.
Let  $T:=\nabla\xi-\beta(\mathbb I_{TM}-  \xi\otimes \eta) $ and
 $\alpha:=||T||/\sqrt{m-1}$ .
 Then  we have
\begin{enumerate}
\item[(a)]  $\xi\beta+k+\beta^2=\alpha^2$
\item[(b)] $d\alpha+2\alpha\beta\eta=0$
\item[(c)] $d\beta=(\xi\beta)\eta=(\alpha^2-k-\beta^2)\eta$. 
\end{enumerate}

In particular, if $\alpha\neq0$ on $M$, then $\beta=0$ and  $k=\alpha^2$ is constant.
Moreover  $(\psi,\xi, \eta, \la,\ra)$  is an $\alpha$-Sasakian structure on $M$,
where $\psi=-\alpha^{-1}T$.
\end{theo}
\begin{proof}
By the hypothesis,  we can see that $T$ is skew-symmetric and $T\xi=0$.
Further, for any $X$, $Y$,  $Z\in TM$,
we have
\begin{align*}
(\nabla _{X}\nabla _{Y}-\nabla_{\nabla_XY})\xi	=
&(\nabla_XT)Y+(X\beta)\bar Y-\beta\la TX,Y\ra\xi \notag	\\
&	 -\beta^2\la \bar X,Y\ra\xi 
	-\beta\eta(Y) TX-\beta^2\eta(Y)\bar X.	%\label{la1.3}
\end{align*}
Hence, from (\ref{la1.1}) and the above equation, we calculate
\begin{align}
&\theta(X)\la  \bar Y,Z\ra 	-\theta(Y)\la \bar Z,X\ra													
+\la(\nabla_XT)Y,Z\ra+\la(\nabla_YT)Z,X\ra	\notag\\
  &+\beta\{\eta(X)\la TY,Z\ra+\eta(Y)\la TZ,X\ra-2\eta(Z)\la TX,Y\ra\}=0 \label{la1.5} 
\end{align}
where $\theta:=d\beta+(k+\beta^2)\eta$.
First, taking cyclic sum over $X,\,Y$ and $Z$ in the above equation, and then compare the obtained equation  with (\ref{la1.5}),
we obtain 
%\[
%\la(\nabla_XT)Y,Z\ra+\la(\nabla_YT)Z,X\ra+\la(\nabla_XT)X,Y\ra=0.
%\]
%Hence,  (\ref{la1.5}) reduces to
\begin{align} \label{la1.7b}
&(\nabla_ZT)X=\theta(X)\bar  Z-\la \bar Z, X\ra\theta^\sharp
  +\beta\{-\eta(X) TZ+\la TZ,X\ra\xi-2\eta(Z)TX\}.
\end{align}
By putting $X=\xi$ in (\ref{la1.7b}), we obtain
\begin{align} \label{la1.7c}
T^2=\{\xi\beta+k+\beta^2\}(\mathbb I_{TM}-  \xi\otimes \eta).
\end{align}
Hence, Statement (a)  can be deduced from this equation.

If $\alpha=0$, then Statement (b) is trivial. 
Note that $T=0$ in this case. Since $\dim M\geq 3$, 
we can  derive immediately from (\ref{la1.7b}) that 
%\[
%\theta(X)\bar  Z-\la \bar Z, X\ra\theta^\sharp=0, \quad X,Z\in TM.
%\]
%Since $\dim M\geq 3$, this equation gives 
$\theta=0$ and so Statement (c) is obtained.

Next, we consider  $\alpha\neq0$.
Let $\psi=-\alpha^{-1}T$. Then by virtue of  (\ref{la1.7c}), we can verify that $(\psi,\xi,\eta,\la,\ra)$ is an almost contact metric structure on $M$.
Next (\ref{la1.7b}) gives
\begin{align} \label{eqn:nabla_phi}
-(Z\alpha)\psi X-\alpha (\nabla_Z\psi)X
=\theta(X)\bar Z-\la \bar Z, X\ra\theta^\sharp&		\notag\\
 +\alpha\beta\{\eta(X)\psi Z-\la \psi Z,X\ra\xi+2\eta(Z)\psi X\},& \quad X,Z\in TM.	
\end{align}
Now consider $X\perp \xi$,  after taking inner product with $\psi X$ on both sides of this equation,  yields 
\[
(X\beta)\psi X-(\psi X\beta) X+\la X,X\ra\{\grad \alpha+2\alpha\beta\xi\}=0.	
\]
Since $\dim M\geq 5$,  the above equation gives Statement (b) and 
$d\beta=(\xi\beta)\eta$. Hence we obtain Statement (c) and $\theta=\alpha^2\eta$.
It follows that (\ref{eqn:nabla_phi}) can be simplified to
\begin{align*}% \label{eqn:nabla_phi-2}
\alpha (\nabla_Z\psi)X
=&\alpha^2\{\la  Z, X\ra\xi	-\eta(X) Z\}	%	\notag\\
   +\alpha\beta\{-\eta(X)\psi Z+\la\psi Z,X\ra\xi\}, \quad X,Z\in TM.	
\end{align*}
As  $\alpha\neq0$,  the above equation implies that $(\psi,\xi,\eta,\la,\ra)$ is an $(\alpha,\beta)$-trans-Sasakian structure on $M$. Since an $(\alpha,\beta)$-trans-Sasakian manifold of dimension greater than 3 is either  
$\alpha$-Sasakian or $\beta$-Kenmotsu (cf. \cite{marrero}) and $\alpha\neq0$, we conclude that 
$\beta=0$ and so $(\psi,\xi,\eta,\la,\ra)$ is an $\alpha$-Sasakian structure on $M$. 
It further follows from Statements (a)--(b) that $k=\alpha^2$ is a constant.
\end{proof}
 
\begin{rem}
For $\beta=0$, Theorem~\ref{theo:Nk-Sasakian} also holds  for  $m=3$ and is reduced to the characterizations of $\alpha$-Sasakian manifolds given in \cite{boyer,okumura}. 
A similar result was also obtained in \cite{ghosh} under the setting of generalized Sasakian space form.
\end{rem}

\begin{theo} \label{theo:pre-conformal}
Let $M$ be an almost contact metric manifold of dimension $2n+1\geq 5$. Then  
$M$ is an $N(k)$-manifold  and   $\nabla \xi =\beta(\mathbb I_{TM}-  \xi\otimes \eta)$ if and only if   $M$ is locally a warped product 
$\mathbb R\times_{b}P$ with $\ddot b=-kb $ and $\beta=\dot b/b$, where $P$ is an almost Hermitian manifold.
\end{theo}
\begin{proof}
By the hypothesis, we see that $\nabla_{\xi}\xi=0$. Hence both distributions $D$ and $\mathbb R\xi$ are integrable. 
Moreover,   $D$ is totally umbilical  and $\mathbb R\xi$ is autoparallel.  By Theorem~\ref{theo:Nk-Sasakian}(c), $\beta$ is constant along $D$, 
so $D$ is spherical. 
As a result, 
 $M$ is locally a warped product of $\mathbb R$ and a leaf $P$ of $D$ 
(cf. \cite{reckziegel-schaaf}). Since $P$ is an invariant submanifold in $M$ normal to $\xi$, it admits an almost Hermitian structure. 

Conversely, suppose $M=\mathbb R\times_b P$, where $P$ is an almost Hermitian manifold . Then it follows from 
(\ref{eqn:R_warped-product}) that  
\[
R(X,Y)\xi=k\{\eta(Y)X-\eta(X)Y\}, \quad k=-\frac{\ddot b}b.
\] 
Next,  by virtue of \cite[Proposition 35]{oneill}, we see that  $\nabla_{\xi}\xi=0$ and 
$\nabla_X\xi=(\xi b/b)X=(\dot b/b)X$, for $X\in D$. These imply that 
$\nabla \xi =\beta(\mathbb I_{TM}-  \xi\otimes \eta)$,  where $\beta=\dot b/b$,  and the proof is completed.
\end{proof}

%\begin{cor} \label{cor:riemannian-product}
%Let $M$ be an almost contact metric manifold of dimension $2n+1\geq 5$. Then  
%$\nabla  \xi =0$ if and only if   $M$ is locally a Riemannian product  
%$\mathbb R\times P$, where $N$ is an almost Hermitian manifold.
%\end{cor}

By Theorem~\ref{theo:pre-conformal} and (\ref{eqn:R_warped-product}), we can easily obtain the following local characterization for the class of generalized Sasakian space form given in Theorem~\ref{theo:warped-product-gSSF}.

\begin{cor} \label{cor:pre-warped_prod_gSSF}
Let $M^{2n+1}(f_1,f_2,f_3)$ be a generalized Sasakian Space form, $n\geq2$. Then  
$\nabla \xi =\beta(\mathbb I_{TM}-  \xi\otimes \eta) $ if and only if   $M$ is locally a warped product 
$\mathbb R\times_{b}P$ with $\ddot b=-kb $ and $\beta=\dot b/b$, where $P$ is a generalized complex space form with functions
$F_1=b^2f_1+\dot b^2$ and $F_2=b^2f_2$.
\end{cor}

%\begin{cor} \label{cor:riemannian-product_gSSF}
%Let $M^{2n+1}(f_1,f_2,f_3)$ be a generalized Sasakian Space form, $n\geq 2$. Then  
%$\nabla \xi =0$ if and only if   $M$ is locally Riemannian product  
%$\mathbb R\times P$, where $P$ is a generalized complex space form with functions
%$F_1=f_1$ and $F_2=f_2$.
%\end{cor}

%%%%%%%%%%%%%%%%%%%%%%%%%%%%%%%%%%%%%%%%%%%%%%%%%%%%%%%%%%%%%%%%%%%%%%
%%%%%%%%%%%%%%%%%%%%%%%%%%%%%%%%%%%%%%%%%%%%%%%%%%%%%%%%%%%%%%%%%%%%%%
\section{Some identities of generalizes Sasakian space forms}
Let $M^{2n+1}(f_1,f_2,f_3)$  be a generalized Sasakian space form, $n\geq2$. By considering the Bianchi's second identity
\[
\mathfrak S_{X,Y,Z}(\nabla_X R)(Y,Z)W=0
\]
where $\mathfrak S$ represents the cyclic sum over $X$, $Y$ and $Z$, we obtain
\begin{align}\label{eqn:pre-600}
\mathfrak S_{X,Y,Z}(A_1+A_2+A_3+A_4+A_5+A_6)=0
\end{align}
where 
\begin{align*}
A_1=&(Xf_1)\{\la Z,W\ra Y-\la Y,W\ra Z\} \\
A_2=&(Xf_2)\{\la\phi Z,W\ra\phi Y-\la\phi Y,W\ra\phi Z-2\la\phi Y,Z\ra\phi  W\}\\
A_3=&(Xf_3)\{\eta(W)\eta(Y)Z-\eta(W)\eta(Z)Y-\la Z,W\ra\eta(Y)\xi+\la Y,W\ra\eta(Z)\xi\}	\\
A_4=&f_2\{\la(\nabla_X\phi) Z,W\ra\phi Y-\la (\nabla_X\phi) Y,W\ra\phi Z-2\la(\nabla_X\phi) Y,Z\ra\phi  W	\\
					&  +       \la\phi Z,W\ra(\nabla_X\phi) Y-\la\phi Y,W\ra(\nabla_X\phi) Z-2\la\phi Y,Z\ra(\nabla_X\phi)  W\}\\
A_5=&f_3\{\la\nabla_X\xi,W\ra\eta(Y)Z-\la\nabla_X\xi,W\ra\eta(Z)Y-\la Z,W\ra\eta(Y)\nabla_X\xi+\la Y,W\ra\eta(Z)\nabla_X\xi\}	\\
A_6=&f_3\{\eta(W)d\eta(X,Y)Z-\la Z,W\ra d\eta(X,Y)\xi\}.
\end{align*}

Observe that 
\begin{align*}
\mathfrak S_{X,Y,Z}A_5
=&f_3\mathfrak S_{X,Y,W}\la\nabla_X\xi,W\ra\{\eta(Y)\bar Z-\eta(Z)\bar Y\}	\\
  &+f_3\mathfrak S_{X,Y,Z}\la\bar X,W\ra\{\eta(Y)\nabla_Z\xi-\eta(Z)\nabla_Y\xi\}
\end{align*}
\begin{align*}
\mathfrak S_{X,Y,Z}(A_3+A_6)
=\mathfrak S_{X,Y,Z}\{(df_3\wedge\eta)(X,Y)+f_3d\eta(X,Y)\}\{\eta(W)\bar Z-\la \bar Z,W\ra\xi\}  
\end{align*}
\begin{align*}
\mathfrak S_{X,Y,Z} A_1
=&\mathfrak S_{X,Y,Z}(Xf_1)\{\la\bar Z,W\ra \bar Y-\la\bar Y,W\ra\bar Z
      +\la\bar Z,W\ra \eta(Y)\xi-\la\bar Y,W\ra\eta(Z)\xi	\\
	& +\eta(Z)\eta(W)\bar Y-\eta(Y)\eta(W)\bar Z\} 	\\
=&\mathfrak S_{X,Y,Z}(Xf_1)\{\la\bar Z,W\ra \bar Y-\la\bar Y,W\ra\bar Z\}\\
 &   -\mathfrak S_{X,Y,Z}(df_1\wedge\eta)(X,Y)\{\eta(W)\bar Z-\la \bar Z,W\ra\xi\}.
\end{align*}
It follows that 
\begin{align*}
\mathfrak S_{X,Y,Z}(A_1+A_3+A_6)
=\mathfrak S_{X,Y,Z}(Xf_1)\{\la\bar Z,W\ra \bar Y-\la\bar Y,W\ra\bar Z\}\\
+\mathfrak S_{X,Y,Z}(-dk\wedge\eta+f_3d\eta)(X,Y)\{\eta(W)\bar Z-\la \bar Z,W\ra\xi\}
\end{align*}
where $k=f_1-f_3$.
By substituting into (\ref{eqn:pre-600}), we obtain
\begin{align}\label{eqn:600}
&\mathfrak S_{X,Y,Z}\Big((Xf_1)\{\la\bar Z,W\ra \bar Y-\la\bar Y,W\ra\bar Z\} \nonumber\\
&+(Xf_2)\{\la\phi Z,W\ra\phi Y-\la\phi Y,W\ra\phi Z-2\la\phi Y,Z\ra\phi  W\}\nonumber\\
&+(-dk\wedge\eta+f_3d\eta)(X,Y)\{\eta(W)\bar Z-\la \bar Z,W\ra\xi\} \nonumber\\
&+f_3\la\nabla_X\xi,W\ra\{\eta(Y)\bar Z-\eta(Z)\bar Y\}	%	\nonumber	\\
   +f_3\la\bar X,W\ra\{\eta(Y)\nabla_Z\xi-\eta(Z)\nabla_Y\xi\}	\nonumber\\
&+f_2\{\la(\nabla_X\phi) Z,W\ra\phi Y-\la (\nabla_X\phi) Y,W\ra\phi Z-2\la(\nabla_X\phi) Y,Z\ra\phi  W	\nonumber\\
& +        \la\phi Z,W\ra(\nabla_X\phi) Y-\la\phi Y,W\ra(\nabla_X\phi) Z-2\la\phi Y,Z\ra(\nabla_X\phi)  W\}\Big)=0
\end{align}
for any $X$, $Y$, $Z$, $W\in TM$.

By putting $W=\xi$ in (\ref{eqn:600}), we have 
\begin{align}\label{eqn:00}
&\mathfrak S_{X,Y,Z}\Big((-dk\wedge\eta+f_3d\eta)(X,Y)\bar Z \nonumber\\
&+f_2\{-\la \nabla_X\xi,\phi Z\ra\phi Y+\la \nabla_X\xi,\phi Y\ra\phi Z+2\la\phi Y,Z\ra\phi\nabla_X\xi\}\Big)=0.
\end{align}
If we put $Z=\xi$ and $X$, $Y\in   D$ in (\ref{eqn:00}), then it becomes
\begin{align*}
&-\{Yk+f_3V^\flat(Y)\}X+\{Xk+f_3V^\flat(X)\}Y\\
&+f_2\{-V^\flat(\phi Y)\phi X+V^\flat(\phi X)\phi Y+2\la\phi X,Y\ra\phi V\}=0.
\end{align*}
By  suitable choices of $X$ and $Y$ in  the above equation, gives
\begin{align}
f_2V&=0 \label{eqn:10} \\
dk+f_3V^\flat&=(\xi k)\eta.  \label{eqn:20}
\end{align}
It follows from  (\ref{eqn:00}) and (\ref{eqn:20}) that 
\begin{align}\label{eqn:30}
&\mathfrak S_{X,Y,Z}\Big(f_3d\eta( X,  Y)\la Z,W\ra
				+f_2\{-\la \nabla_X\xi,\phi Z\ra\la \phi Y,W\ra \nonumber\\
&+\la \nabla_X\xi,\phi Y\ra\la\phi Z,W\ra+2\la\phi Y,Z\ra\la\phi\nabla_X\xi,W\ra\}\Big)=0, \quad 
X,Y,Z,W\in D.
\end{align}
Contraction at  $X$ and $W$ over a local orthonormal frame on $D$, gives  
\begin{align}\label{eqn:40}
\{3f_2+(2n-2)f_3\}d\eta(Y,Z)-f_2d\eta(\phi Y,\phi Z)+2f_2\delta\Phi(\xi)\la\phi Y,Z\ra=0, \quad Y,Z\in D
\end{align}
where we have used 
%where $\diver\phi=\sum^{2n}_{j=1}(\nabla_{e_j}\phi)e_j+(\nabla_{\xi}\phi)\xi$. 
$\sum_{j=1}^{2n}d\eta(\phi E_j, E_j)=2\delta\Phi(\xi)$, where  $\{E_1,\cdots,E_{2n+1}\}$ is a local orthonormal frame  on $TM$. The above equation implies that  
\begin{align}
\{2f_2+(n-1)f_3\}\{d\eta(Y,Z)-d\eta(\phi Y,\phi Z)\}&=0, \quad Y,Z\in   D \label{eqn:60} \\
(f_2-f_3)\delta\Phi(\xi)&=0. \label{eqn:50}
\end{align}

\begin{lem}\label{lem:f_3_neq_0}
Let $M^{2n+1}(f_1,f_2,f_3)$ be a generalized Sasakian space form,  $n\geq 2$. 
Suppose $f_2=0$ and $f_3\neq0$. Then 
\begin{itemize}
\item[(a)] $df_3=f_3V^\flat+(\xi f_3)\eta$
\item[(b)] $\nabla\xi=\beta(\mathbb I_{TM}-  \xi\otimes \eta)+ V\otimes \eta$
\item[(c)] $df_1=-2\beta f_3\eta$
 \end{itemize}
where $\beta=-(1/2n)\delta\eta$.
\end{lem}
\begin{proof}
Under the hypothesis,  (\ref{eqn:40}) gives $d\eta=\eta\wedge V^\flat$.  This  fact and  (\ref{eqn:20}) imply that   $-dk\wedge\eta+f_3d\eta=0$.
Fixed a unit vector $X\in D$ and select $Y$, $Z\in D$  such that  $X$, $Y$, $Z$ are orthonormal. Then  (\ref{eqn:600})  gives
$(Xf_1)Z-(Zf_1)X=0$ and hence %$Xf_1=0$, for any $X\in   D$. In other words, 
\begin{align}\label{eqn:df1}
df_1=(\xi f_1)\eta.
\end{align}
By substituting this  into (\ref{eqn:20}), we get $df_3=f_3 V^\flat+(\xi f_3)\eta$. Further, (\ref{eqn:600}) reduces to
\begin{align*}
&(\xi f_1)\mathfrak S_{X,Y,Z}\eta(X)\{\la\bar Z,W\ra \bar Y-\la\bar Y,W\ra\bar Z\}
+f_3\mathfrak S_{X,Y,W}\la\nabla_X\xi,W\ra\{\eta(Y)\bar Z-\eta(Z)\bar Y\}		\nonumber	\\
&+f_3\mathfrak S_{X,Y,Z}\la\bar X,W\ra\{\eta(Y)\nabla_Z\xi-\eta(Z)\nabla_Y\xi\}=0
\end{align*}
for any $X$, $Y$, $Z$, $W\in TM$.
If we put $X=\xi$ and $Y$, $Z$, $W\in   D$ in this equation,  then
\begin{align}\label{eqn:340}
(\xi f_1)\{\la Z,W\ra  Y-\la Y,W\ra  Z\ra\}
			+f_3\{\la\nabla_Z\xi,W\ra  Y-\la\nabla_Y\xi,W\ra  Z	\nonumber\\
-\la Y,W\ra \nabla_Z\xi +\la Z,W\ra \nabla_Y\xi\}=0.
\end{align}
By contracting $Z$ and $W$ over a local orthonormal frame on $D$,   we get 
\begin{align}
(2n-2)f_3\nabla_Y\xi+\{(2n-1)\xi f_1-f_3\delta\eta\}Y=0, \quad Y\in   D.
\end{align}
Since $f_3\neq0$, the above equation deduces that  
$\nabla\xi=\beta(\mathbb I_{TM}-  \xi\otimes \eta)+V \otimes\eta$  and $\xi f_1+2\beta f_3=0$, 
where $\beta=-(1/2n)\delta\eta$.
This, together with (\ref{eqn:df1}), give Statement (c).
\end{proof}

Now we choose unit vectors $Y,Z\in D$ such that $Z\perp Y,\phi Y$. 
By putting  $X=\phi Y$ and $W=Z$ in (\ref{eqn:600}), we obtain
%\includegraphics[width=2cm]{"./question cartoon"}\\
% % % % % % % % % % % % % % % % % % % % % % % % % % % % % % % % % % % % % % % % % % % % % % % % % % % % % %
% % % % % % % % % % % % % % % % % % % % % % % % % % % % % % % % % % % % % % % % % % % % % % % % % % % % % % % % % % % % % % % % % % % % % % % % % % % % % % % % % % % % % % % % % % % % % % % % % % % % % % % % % % % % % % % % % % % % % % % % % % % % % % % % % % % % % % % % % % % % % % % % % % % % % % % % % % % % % % % % % % % % % % % % % % % % % % % % % % % % % % % % % % % % % % % % % % % % % % % % % % % % % % % % % % % % % % % % 
%Consider an orthonormal basis  $\{E_1, E_2, E_3=\phi E_1, E_4=\phi E_2\}$ of $D$. 
%Fixed $j,k\in\{1,2,3,4\}$ with $k\neq j(\mod 2)$.
%By putting  $X=\phi E_j$, $Y=E_j$ and $Z=W=E_k$  $(\perp E_j,\phi E_j)$ in (\ref{eqn:600}), we get
\begin{align*}
0=&(\phi Yf_1)Y-(Yf_1)\phi Y+2(Z f_2)\phi Z\\
		&+f_2\{3\la(\nabla_{\phi Y}\phi)\phi Y, \phi Z\ra\phi Z+3\la(\nabla_{Y}\phi) Y,\phi Z\ra\phi Z\\
		&+\la(\nabla_{Z}\phi)Z, Y\ra Y+\la(\nabla_{Z}\phi)Z, \phi Y\ra\phi Y
					+2(\nabla_{Z}\phi)Z\}\quad  (\mod \xi)  %\\
	%=&(\phi E_jf_1)E_j-(E_jf_1)\phi E_j-2(E_k f_1)\phi E_k
	%+3(1-f_1)\{\la(\nabla_{E_j}\phi) E_j,\phi E_k\ra\phi E_k\\
	%	&+\la(\nabla_{\phi E_j}\phi)\phi E_j, \phi E_k\ra\phi E_k			+(\nabla_{E_k}\phi)E_k\}\quad  (\mod \xi)					
\end{align*}
where we have used the identity
\begin{align}\label{eqn:iden}
(\nabla_X\phi)\phi+\phi(\nabla_X\phi)=(\nabla_X\xi)\otimes\eta+\xi\otimes(\nabla_X\eta), \quad X\in TM.
\end{align}
We can further deduce that  
\begin{align*}
2Z f_2=&-3f_2\{\la(\nabla_{\phi Y}\phi)\phi Y, \phi Z\ra+\la(\nabla_{Y}\phi) Y,\phi Z\ra\}  \\
Yf_1=&3f_2\la(\nabla_{Z}\phi)Z, \phi Y\ra   \\
2f_2(\nabla_{Z}\phi)Z
		=&-\{\phi Yf_1+f_2\la(\nabla_{Z}\phi)Z, Y\ra\}Y	\\
			&+\{Yf_1-f_2\la(\nabla_{Z}\phi)Z, \phi Y\ra\}\phi Y\quad  (\mod \xi)  		
\end{align*}
for any unit vectors $Y,Z\in D$ with $Z\perp Y,\phi Y$.

We note that these equations also hold if we switch $Y$ and $Z$. Hence we obtain
\begin{align}
Z f_1=&3f_2\{\la(\nabla_{Y}\phi)Y, \phi Z\ra=-Zf_2  \label{eqn:741}\\
3f_2(\nabla_{Z}\phi)Z
		=&-(\phi Yf_1)Y+(Yf_1)\phi Y\quad (\mod \xi) \nonumber	\\
		=&(\phi Yf_2)Y-(Yf_2)\phi Y\quad  (\mod \xi)  		 \label{eqn:742}
\end{align}
for any unit vectors $Y,Z\in D$ with $Z\perp Y,\phi Y$.

%%%%%%%%%%%%%%%%%%%%%%%%%%%%%%%%%%%%%%%%%%%%%%%%%%%%%%%%%%%%%%%%%%%%%%
%%%%%%%%%%%%%%%%%%%%%%%%%%%%%%%%%%%%%%%%%%%%%%%%%%%%%%%%%%%%%%%%%%%%%%
\section{Riemannian manifolds of quasi-constant sectional curvatures and canal hypersurfaces}
We shall first review some results in (space-like) canal hypersurfaces in Euclidean and Minkowski spaces. 
The main references of this section are \cite{ganchev,ganchev2}.

A canal hypersurface $M$ in the Euclidean space $\mathbb R^{m+1}$ is the envelope of a one-parameter family of hyperspheres 
$\{S^{m}(s)$, $s\in I\subset \mathbb R\}$, given by the following conditions
\begin{align*}
\la \varphi-\gamma(s),\varphi-\gamma(s)\ra&=r^2(s), \quad r(s)>0 \\
\la \varphi-\gamma(s),\dot\gamma(s)\ra      &=-r(s)\dot r(s)
\end{align*}
where $\gamma(s)$ and $r(s)$ are the center and radius of the corresponding sphere $S^{m}(s)$ respectively, 
$\varphi=\varphi(s,u_1,\cdots,u_{m-1})$ is the position vector of  $M$, 
$\dot\gamma=\partial_s \gamma$ and $\dot r=\partial_s r$.
We suppose that the curve $\gamma$ is parametrized by a natural parameter $s$.

Similarly, a space-like canal hypersurface $M$ in the Minkowski space $\mathbb R^{m+1}_1$ is the envelope of a one-parameter family of space-like hyperspheres $\{S^{m}(s)$, $s\in I\subset \mathbb R\}$, 
given by the conditions:
\begin{align*}%\label{eqn:canal1}
\la \varphi-\gamma(s),\varphi-\gamma(s)\ra&=-r^2(s), \quad r(s)>0 \\
\la \varphi-\gamma(s),\dot\gamma(s)\ra      &=r(s)\dot r(s)
\end{align*}
where $\gamma(s)$ and $r(s)$ are the center and radius of the corresponding sphere $S^{m}(s)$ respectively, 
and $\varphi=\varphi(s,u_1,\cdots,u_{m-1})$ is the position vector of  $M$. 

The space-like canal hypersurface $M$ in $\mathbb R^{m+1}_1$  is said to be of elliptic, hyperbolic or parabolic type if the curve $\gamma$ is time-like, space-like or light-like respectively.
We suppose that $s$ is a natural parameter of the curve $\gamma$ for a space-like canal hypersurface of elliptic or hyperbolic type.

In the following, we will use a unified notation. 
Let $M^{m+1}(c)$ denote the Euclidean space $\mathbb R^{m+1}$ (resp. Minkowski space $\mathbb R^{m+1}_1$) for $c=1$ (resp. $c=-1$) .

We will use the notation $M(c,\varepsilon)$ to denote a (space-like) canal hypersurface in $M^{m+1}(c)$, 
where $\varepsilon=\la \dot\gamma,\dot\gamma\ra$, that is, 
\begin{itemize}
\item for $c=1$:  we put $\varepsilon=1$; and 
\item for $c=-1$:  we put 
$\varepsilon=\left\{\begin{array}{rl}
  -1, &  \text{when  $M(c,\varepsilon)$ is of elliptic type}\\
   1, &  \text{when  $M(c,\varepsilon)$ is of hyperbolic type}\\
		0, &  \text{when  $M(c,\varepsilon)$ is of parabolic type}.
	\end{array}\right.
	$
\end{itemize}
It follows that  $M(c,\varepsilon)$ satisfies the following conditions
\begin{align}
\la \varphi-\gamma(s),\varphi-\gamma(s)\ra&=cr^2(s), \quad r(s)>0  \label{eqn:canal1}\\
\la \varphi-\gamma(s),\dot\gamma(s)\ra      &=-cr(s)\dot r(s). \label{eqn:canal2}
\end{align}
We consider the unit normal vector field 
\[
N=-\frac{\varphi-\gamma}r
\]
to $M(c,\varepsilon)$. Then we can see that 
\[
\la N,N\ra=c, \quad \la N,\dot\gamma\ra=c\dot r.
\]
Next  define a unit tangent vector field $\xi$ by
\[
\xi=\frac{\dot \gamma-\dot r N}{\sqrt{\varepsilon-c\dot r^2}}.
\]
It follows that the distribution $D=\mathbb R\xi^\perp$ is spanned by $\{\partial\varphi/\partial_{u_1},\cdots,\partial\varphi/\partial_{u_{m-1}}\}$.
Let $\bar\nabla$ be the Levi-Civita connetion of $M^{m+1}(c)$. Then 
\begin{align}\label{eqn:canal_normal}
\bar\nabla_XN&=-\lambda X;\quad  \bar\nabla_{X}\xi=\beta X; \quad  \lambda=\frac1r,   
																										\quad \beta=\frac{\dot r}{r\sqrt{\varepsilon-c\dot r^2}} \\
\bar\nabla_\xi N&=-\nu\xi;  \quad \nu=\frac1r-\frac{\sqrt{\varepsilon-c\dot r^2}}{r\la\xi,\dot\varphi\ra}. \nonumber  
\end{align}
It follows that the shape operator  $A$ of $M(c,\varepsilon)$ is given by
\[
A\xi=\nu\xi, \quad AX=\lambda X
\]
for any vector $X$ tangent to $D$.
We note that 
\[
c\lambda^2+\beta^2=\frac{c\varepsilon}{r^2(\varepsilon-c\dot r^2)}
\left\{\begin{array}{rl}
>0, & \text{for } c=-1, \varepsilon=-1; \text{or } c=1 \\
<0, & \text{for } c=-1, \varepsilon= 1 \\
=0, & \text{for } c=-1, \varepsilon= 0 \\
\end{array}\right..
\]

We recall that a Riemannian manifold $M$ is said to have quasi-constant sectional curvature if there exists on $M$ a unit vector field $\xi$ such that 
its curvature tensor  $R=f_1R_1+f_3R_3$, for some function $f_1$ and $f_3$ on $M$.

\begin{theo}[\cite{ganchev2}]\label{thm:canal}
Let $M(c,\varepsilon)$ be a (space-like) canal hypersurface in $M^{m+1}(c)$, $m\geq 4$. Then $M(c,\varepsilon)$ is a Riemannian manifold of quasi-constant sectional curvature
with $f_1=c\lambda^2$ and $f_3=c\lambda(\lambda-\nu)$.
In particular, we have
\begin{enumerate}
\item[(a)]  $f_1>0$ for a canal hypersurface in $\mathbb R^{m+1}$;  
\item[(b)]  $f_1<0$ and $f_1+\beta^2>0$ for a space-like canal hypersurface in $\mathbb R^{m+1}_1$ of elliptic type;  
\item[(c)] $f_1+\beta^2<0$ for a space-like canal hypersurface in $\mathbb R^{m+1}_1$ of hyperbolic type;  
\item[(d)] $f_1+\beta^2=0$ for a space-like canal hypersurface in $\mathbb R^{m+1}_1$ of parabolic type.  
\end{enumerate}
Conversely, a Riemannian manifold of quasi-constant sectional curvature with $f_1\neq0$ and $f_3\neq0$ is an open part of one of the above spaces.
\end{theo}

Next, we give an example of Riemannian manifolds  of quasi-constant sectional curvature with $f_1=0$ and $f_3\neq0$.
%Consider a twisted product manifold $M={_a}\mathbb R\times\mathbb R^m$.
%Let $t$ be the standard coordinate of $\mathbb R$. 
%We define $\xi:=a^{-1}\partial_t$ and $\eta:=\xi^\flat$.
 %Further, we denote by $\mathcal L$ 

\begin{theo}\label{thm:TP}
%With the above notations. 
Let $M={_a}\mathbb R\times\mathbb R^{m-1}$ be a twisted prodcut space.  
Then $M$ satisfies %(\ref{eqn:f2=0}), that is,
\begin{align}\label{eqn:f1=f2=0}
R=f_3R_3
%R(X,Y)Z=&f_3\{\eta(Z)\eta(X)Y-\eta(Z)\eta(Y)X-\la Y,Z\ra\eta(X)\xi+\la X,Z\ra\eta(Y)\xi\}	
\end{align}
%for any $X$, $Y$  and $Z\in TM$, 
where $f_3\neq0$ is a function on $M$
if and only if 
		\[
	a(t,x)=\sum^{m-1}_{i=1}(C(t)x_i^2+D_i(t)x_i)+E(t); \quad f_3=\frac{2C}a
	\]
	for some functions  $C\neq0$, $D_1,\cdots, D_{m-1}$ and  $E$ on $\mathbb R$  such that $a>0$. 
\end{theo}
\begin{proof}
Suppose $M={_a}\mathbb R\times\mathbb R^{m-1}$. 
By using the formulas in \cite{garcia}, we see that 
%\begin{align*}
%\nabla_X\xi=\beta(X-\eta(X)\xi)+\eta(X)\{\sigma\xi-\grad\log a\}
%\end{align*}
%for any $X\in TM$, where $\sigma=\xi\log a$ and $\beta=\xi\log b$, and
\begin{align*}%\label{eqn:02a}
R(X,\xi)Z		=&\{(XZ-\nabla_XZ)\log a+(X\log a)(Z\log a)\}\xi 
\end{align*}
for any $X$ and $Z$ orthogonal to $\xi$. 
%, where $h^{\log a}(X,Z)=(XZ-\nabla_XZ)\log a$.
Hence, we can verify that  (\ref{eqn:f1=f2=0}) is equivalent to the  condition
\begin{align}\label{eqn:03b}
f_3\la X,Z\ra		= &(XZ-\nabla_XZ)\log a+(X\log a)(Z\log a) 
\end{align}
for any $X$ and $Z$ orthogonal to $\xi$.

\medskip
\textbf{Sufficiency.}
Suppose the condition (\ref{eqn:f1=f2=0}) is satisfied. 
Let ($x_1,\cdots,x_{m-1})$ be the standard coordinates of $\mathbb R^{m-1}$ and $t$ be the standard coordinate  of $\mathbb R$. 
Then (\ref{eqn:03b})  is represented by
\begin{align}
\partial_i{\partial_j a}&=0, \quad (i\neq j) \label{eqn:ij_a}\\
\partial_i{\partial_i a}&=f_3a. \label{eqn:ii_a}
\end{align}

Solving  (\ref{eqn:ij_a}),  gives   $a=\sum^{m-1}_{j=1}\psi_j(t,x_j)$, 
where each $\psi_j$ depends on $t$ and $x_j$ only. 
By substituting  this into  (\ref{eqn:ii_a}), we obtain 
\begin{align*}%\label{eqn:At}
\partial_i\partial_i\psi_i(t,x_i)=f_3a.
\end{align*}
This implies that each $\partial_i\partial_i\psi_i$ depends only on $t$. We put 
$\partial_i\partial_i\psi_i(t,x_i)=2C(t)$. It follows that 
$\psi_i(t,x_i)=C(t)x_i^2+D_i(t)x_i+E_i(t)$, and so
\begin{align}\label{eqn:a0}
a(t,x)=\sum^{m-1}_{i=1}(C(t)x_i^2+D_i(t)x_i)+E(t), \quad  E(t)=\sum^{m-1}_{i=1}E_i(t), \quad 
f_3=\frac{2C}a.
\end{align}

\medskip
\textbf{Necessity.}
It is directly from the uniqueness of solutions for the PDEs obtained from (\ref{eqn:03b}).
This completes the proof.
\end{proof}

With  the same procedure as in the proof of  Lemma~\ref{lem:f_3_neq_0}, 
except a slight change in the argument while deriving  (\ref{eqn:20}), one may verify that  Lemma~\ref{lem:f_3_neq_0}
holds for Riemannian manifolds of quasi constant sectional curvature as well. 
Accordingly, we state the following lemma without proof.

\begin{lem}\label{lem:f_3_neq_0-b}
Let $M$ be a Riemannian manifold of quasi constant sectional curvature of dimension $m\geq 4$. 
Suppose $f_3\neq0$. Then 
\begin{itemize}
\item[(a)] $df_3=f_3V^\flat+(\xi f_3)\eta$
\item[(b)] $\nabla\xi=\beta(\mathbb I_{TM}-  \xi\otimes \eta)+V\otimes\eta$
\item[(c)] $df_1=-2\beta f_3\eta$
 \end{itemize}
where $\beta=-(1/2n)\delta\eta$.
\end{lem}

\begin{proof}[Proof of Theorem~\ref{thm:quasi}]
We only need to proof the Sufficiency part.
If $f_3=0$, then $M$ is of constant sectional curvature. Next, we obtain (b)--(e) when $f_1\neq0$ and $f_3\neq0$ by virtue of 
Theorem~\ref{thm:canal}.

Finally, we consider $f_1=0$ and $f_3\neq0$.
 By Lemma~\ref{lem:f_3_neq_0-b}(b)--(c), we see that $\nabla_X\xi=0$ for any $X\perp \xi$. 
Hence $D$ is  autoparallel. 
By a result in \cite{reckziegel-schaaf},  $M$ is locally a twisted product $_a\mathbb R\times P$. 
Since $P$ is totally geodesic in $M$ and $f_1=0$, $P$ is flat and so  by Theorem~\ref{thm:TP}, we obtain  (f).
\end{proof}

%%%%%%%%%%%%%%%%%%%%%%%%%%%%%%%%%%%%%%%%%%%%%%%%%%%%%%%%%%%%%%%%%%%%%%
%%%%%%%%%%%%%%%%%%%%%%%%%%%%%%%%%%%%%%%%%%%%%%%%%%%%%%%%%%%%%%%%%%%%%%
\section{Proof of Theorem~\ref{thm:gssf}}
First, if $f_2=f_3=0$, then $M$ is of constant sectional curvature.
Next we consider these cases:
\begin{itemize}
\item $f_1\neq0$, $f_2=0$ and $f_3\neq0$;
\item $f_1=f_2=0$  and $f_3\neq0$;
\item $f_2\neq0$.
\end{itemize}

\medskip
\emph{Case (A)  $f_1\neq0$, $f_2=0$  and $f_3\neq0$.}

By Theorem~\ref{thm:canal}, we see that $M$ is an open part of a (space-like) canal hypersurface in $M^{2n+2}(c)$, listed in  (b)--(e).
We only need to investigate the dimension for canal hypersurfaces in $\mathbb R^{2n+2}$ and space-like canal hypersurfaces of elliptic type in $\mathbb R^{2n+2}_1$.

For a canal hypersurface $M$ in $\mathbb R^{2n+2}$.
It follows from (\ref{eqn:canal1})--(\ref{eqn:canal2}) that each leaf of $D$ lie on an
$2n$-dimensional Euclidean sphere $S^{2n}(s)$ in $\mathbb R^{2n+2}$ (see \cite[pp. 129--130]{ganchev} for detail).
%, centered at $\gamma-r\dot r\dot\gamma$ with radius $r\sqrt{1-\dot r^2}$. 
The almost contact structure of $M$ induces an almost complex structure on $S^{2n}(s)$. 
Since only  a six-dimensional sphere and a  two-dimensional sphere admit an almost complex structure,
%(cf. \cite{borel}),
we conclude that $n=3$ and obtain   (b). 

Next, consider a space-like canal hypersurface of elliptic type in $\mathbb R^{2n+2}_1$.
By using (\ref{eqn:canal1})-(\ref{eqn:canal2}) again, we see that each leaf of $D$ lie on an
$2n$-dimensional Euclidean sphere $S^{2n}(s)$ in $\mathbb R^{2n+2}_1$.
%, centered at $\gamma-r\dot r\dot\gamma$ with radius $r\sqrt{\dot r^2-1}$. 
With a similar argument, we conclude that $n=3$ and obtain   (c).

\medskip
\emph{Case (B)  $f_1=f_2=0$ and $f_3\neq0$.}

By Theorem~\ref{thm:quasi}, we see that  $M$ is locally a twisted product $_a\mathbb R\times P$,
where $P$ is an open part of $\mathbb R^{2n}$.
Since the almost contact metric structure on $M$ induces on each leaf of $D$ an almost Hermitian structure, $P$ is a flat almost Hermitian manifold and  we obtain  (f).

\medskip
\emph{Case (C)  $f_2\neq0$.}

By putting $X=E_j$ and $Z=\phi E_j$ in (\ref{eqn:30}), 
where $\{E_1,\cdots,E_{2n+1}\}$ is a local orthonormal frame on $TM$, and then summing up these equations over $j$,  we have
\begin{align}\label{eqn:120}
f_2\left\{(2n-1)\la\nabla_Y\xi,\phi W\ra-\la\nabla_W\xi,\phi Y\ra-\delta\eta\la\phi Y,W\ra\right\}&	\nonumber\\
+f_3\{d\eta(Y,\phi W)+\delta\Phi(\xi)\la Y,W\ra\}=0,&  \quad Y,W\in   D.
\end{align}
By first switching  $Y$ and $W$ in the above equation and then using the obtained equation and (\ref{eqn:120}), we obtain
\begin{align}\label{eqn:35}
(2n-2)f_2\{2n\la\nabla_Y\xi,\phi W\ra-\delta\eta\la\phi Y,W\ra\}
			+f_3\{2n\delta\Phi(\xi)\la Y,W\ra	&	\nonumber\\
+(2n-1)d\eta(Y,\phi W)-d\eta(\phi Y,W)\}&=0	\\
f_2\{2n\la\nabla_Y\xi,\phi W\ra-2n\la\nabla_W\xi,\phi Y\ra-2\delta\eta\la\phi Y,W\ra\}	&\nonumber \\
+f_3\{d\eta(Y,\phi W)+d\eta(\phi Y,W) \}&=0  \label{eqn:70}
\end{align}
for any $Y$,  $W\in   D$. 
Replacing $W$ by $\phi W$ in  (\ref{eqn:35}), we get
\begin{align*}
-(2n-2)f_2\{2n\la\nabla_Y\xi, W\ra+\delta\eta\la Y,W\ra\}
			-f_3\{2n\delta\Phi(\xi)\la \phi Y,W\ra		\nonumber\\
+(2n-1)d\eta(Y, W)+d\eta(\phi Y,\phi W)\}=0, \quad Y,W\in D.
\end{align*}
It follows from the symmetric part of this equation that 
\[
\la\nabla_Y\xi, W\ra+\la\nabla_W\xi,Y\ra=2\beta\la Y,W\ra, \quad \beta:=-\frac{\delta\eta}{2n}
\]
for any $Y$,  $W\in D$.
On the other hand, it follows from (\ref{eqn:10}) that $V=0$. Hence, we conclude that 
\[
\la\nabla_Y\xi, W\ra+\la\nabla_W\xi,Y\ra=2\beta\la\bar Y,W\ra, \quad Y,W\in TM.
\]

According to Theorem~\ref{theo:Nk-Sasakian}, we have either $\nabla\xi=\beta(\mathbb I_{TM}-  \xi\otimes \eta)$  or 
$\nabla\xi=-\alpha\psi$, where $\alpha\neq 0$ is a constant.
If   $\nabla\xi=\beta(\mathbb I_{TM}-  \xi\otimes \eta)$, then  by virtue of  Corollary~\ref{cor:pre-warped_prod_gSSF}, we see that $M$ is locally a warped product $\mathbb R\times_b P$
with $\ddot b=-kb $ and $\beta=\dot b/b$, where $P$ is a generalized complex space form with functions
$F_1=b^2f_1+\dot b^2$ and $F_2=b^2f_2$.
In particular, for $n\geq 3$,  since $F_2\neq0$, Theorem~\ref{thm:gcsf} tells us that $P$ is a non-flat complex space form.
With the  similar calculations as in the proofs of  
\cite[Proposition 3 and Theorem 4]{kenmotsu}, we see that  $M$ is $\beta$-Kenmotsu. This gives  (g).

Now we suppose that  $\nabla\xi=-\alpha\psi$, where $\alpha$ is a nonzero constant and $(\psi,\xi,\eta,\la,\ra)$ is an $\alpha$-Sasakian structure on $M$.
Note that in this case we have $\beta=-(1/2n)\delta\eta=0$.
Since each $\alpha$-Sasakian structure could be descended to a Sasakian structure after rescaling by a factor $\alpha^2$ on the Riemannian metric
and replacing $\xi$ with $\alpha^{-1}\xi$, without loss of generality, we may assume that $\alpha=1$.

If $\psi=\phi$, then we obtain (h). Hence, in the following we assume $\psi\neq\phi$.
% and $(\psi,\xi,\eta,\la,\ra)$ is a Sasakian structure on $M$.
Fixed $Y$, $Z\in D$, (\ref{eqn:60}) and (\ref{eqn:70}) give
\begin{align}
\{2f_2+(n-1)f_3\}\{\la \phi\psi Y,Z\ra-\la\psi\phi Y,Z\ra\}=0		\label{eqn:60-b} \\
\{nf_2+f_3\}\{\la \phi\psi Y,Z\ra-\la\psi\phi Y,Z\ra\}=0.  \label{eqn:70-b}
\end{align}
 %\begin{lem}
%$\phi\psi=r{\mathbb I}_{TM}-\xi\otimes\eta$.
%\end{lem}
Now we consider two subcases:  $n\geq 3$ and $n=2$.

\medskip
\emph{Case (C-i) $n\geq 3$}.
 
If $n\geq 3$, then (\ref{eqn:60-b}) and (\ref{eqn:70-b}) imply that $\phi\psi=\psi\phi$.
It follows that (\ref{eqn:40}) and (\ref{eqn:35}) become
\begin{align*}%\label{eqn:40}
2\{f_2+(n-1)f_3\}\la\phi\psi Y,Z\ra+f_2\delta\Phi(\xi)\la Y,Z\ra&=0\\
2\{(n-1)f_2+f_3\}\la\phi\psi Y,Z\ra+f_2\delta\Phi(\xi)\la Y,Z\ra&=0.
\end{align*}
Since $f_2\neq0$, either $f_2+(n-1)f_3\neq0$ or $(n-1)f_2+f_3\neq0$. We can then deduce from these equations that 
%$\la \phi\psi Y,Z\ra=\la \psi\phi Y,Z\ra=r\la Y,Z\ra$. 
%This implies that 
$\phi=\psi$; a contradiction. Hence this case can not occur.

\medskip
\emph{Case (C-ii) $n=2$}.

%In this case, (\ref{eqn:40})--(\ref{eqn:60}) and   (\ref{eqn:35})  give  
 %\begin{align}
%\{3f_2+2f_3\}\la\phi\psi Y,Z\ra-f_2\la  \psi \phi Y, Z\ra &=-f_2\frac{\delta\Phi(\xi)}{\alpha}\la Y,Z\ra \label{eqn:40-c}\\
%\{2f_2+f_3\}\{\la\phi\psi Y,Z\ra-\la  \psi \phi Y, Z\ra\}&=0 \label{eqn:60-c}\\
%\{4f_2+3f_3\}\la\phi\psi Y,Z\ra+f_3\la  \psi \phi Y, Z\ra&=-f_3\frac{\delta\Phi(\xi)}{\alpha}\la Y,Z\ra \label{eqn:35-c}
%\end{align}
%for any $Y$ and $Z\in   D$. 
We consider two subcases: $\phi\psi=\psi\phi$ and $\phi\psi\neq\psi\phi$.

\medskip
\emph{Case (C-ii-a) $\phi\psi=\psi\phi$}.

By (\ref{eqn:40}), we see that 
$
-2\{f_2+f_3\}\psi+f_2{\delta\Phi(\xi)}\phi =0. %\label{eqn:40-d}
$
If $f_2\neq -f_3$, then this again gives $\phi=\psi$ and so we obtain $f_2+f_3=0$.

On the other hand, recall that the Ricci tensor $S$ for a generalized Sasakian space form $M$ is given by
\begin{align*}
S=\{2nf_1+3f_2-f_3\}\mathbb I_{TM}-\{3f_2+(2n-1)f_3\} \xi\otimes \eta.
\end{align*}
Since $n=2$, and $f_2+f_3=0$ in our case, $(\psi,\xi,\eta,\la,\ra)$ is a Sasakian Einstein  structure on $M$. 
This gives (i).
%By a well-known fact that complete Einstein Sasakian manifolds are compact and $M$ is not of  constant sectional curvature, 

\medskip
\emph{Case (C-ii-a) $\phi\psi\neq \psi\phi$}.

It follows from   (\ref{eqn:50}) and (\ref{eqn:70-b})  that $2f_2+f_3=\delta\Phi(\xi)=0$.
With the help of these results, we can further obtain  $\phi\psi=-\psi\phi$ from  (\ref{eqn:40}). 

By using (\ref{eqn:741}) and the facts $2f_2+f_3=0$, $f_1-f_3=1$, we obtain $Zf_1=0$ for $Z\in D$.
Since $(M,\psi,\xi,\eta)$ is a Sasakian manifold, $f_1$ is a constant (and so are both $f_2$ and $f_3$). 
By using the fact $\dim D=4$,  (\ref{eqn:iden}) and (\ref{eqn:742}), we have
\[(\nabla_{X}\phi)Y\perp D, \quad X,Y\in D.
\]
It follows that   $M$ is locally $\psi$-symmetric (in the sense of \cite{blair-vanhecke}). 
Hence, the universal covering $\bar M$ of $M$ is a five-dimensional naturally reductive homogeneous space (cf. \cite{blair-vanhecke}) 
and so $\bar M$ is one of the spaces listed in 
\cite{kowalski}. 
Since $(\psi,\eta,\xi,\la,\ra)$, is not a Sasakian space form structure on $\bar M$,
we  conclude that $\bar M$ is one of the  homogeneous spaces of Type I or Type II listed in \cite{kowalski}. 

Let  $\tilde\nabla$ be the Okumura connection on $(\bar M,\psi,\eta,\xi,\la,\ra)$, that is,  (cf. \cite{okumura2})
\[
\tilde \nabla_XY=\nabla_XY+\frac12\tilde T(X,Y)
\]
where $\tilde T$ is the torsion of $\tilde \nabla$ given by
\begin{align*}%\label{eqn:tilde-T}
\tilde T(X,Y)=2\{-\la\psi X,Y\ra\xi-\eta(X)\psi Y+\eta(Y)\psi X\}.
\end{align*}
Denote by $\tilde R$ the curvature tensor of $\tilde \nabla$. Then 
\begin{align}\label{eqn:tilde_R}
\tilde R(X,Y)Z=&R(X,Y)Z+\frac14\{\tilde T(Y,\tilde T(X,Z))-\tilde T(X,\tilde T(Y,Z))  \notag\\
													&													+2\tilde T(\tilde T(X,Y),Z)\}.
\end{align}

Now we express  $\bar M=G/H$ as a homogeneous Riemannian space. 
Let $\mathfrak g$ and $\mathfrak h$ be the Lie algebra of  $G$ and $H$ respectively.
We consider an $\Ad (H)$-invariant decomposition $\mathfrak g=\mathfrak h\oplus \mathfrak m$ such that
the homogeneous space structure is naturally reductive
and  its  corresponding canonical connection  coincides with  the Okumura connection $\tilde \nabla$.
Up to identification, at the origin $o\in \bar M$, its tangent space $T_o\bar M=\mathfrak m$ and  we have 
\begin{align*}
\tilde T(X,Y)_o   &=-[X,Y]_{\mathfrak m} \\%\label{eqn:tilde-T}\\
(\tilde R(X,Y)Z)_o &  =-[[X,Y]_{\mathfrak h},Z] %\label{eqn:tilde-R2}
\end{align*}
for any $X$, $Y$, $Z\in\mathfrak m$.

By adopting the results obtained in \cite{kowalski-vanhecke, kowalski}, we may consider an orthonormal basis 
$\{X_1,X_2=\psi X_1,X_3,X_4=\psi X_3, X_5=-\xi\}$ of 
$\mathfrak m$ such that 
\begin{align}\label{eqn:tilde-R2}
\tilde R(X_1,X_3)=\tilde R(X_1,X_4)=\tilde R(X_2,X_3)=\tilde R(X_2,X_4)=0
\end{align} 
%\begin{align*}
%&[X_1,X_2]=-2X_5-uP, \quad  [X_3,X_4]=-2X_5-vP 	\\
%&[X_1,X_5]=2X_2, \quad [X_2,X_5]=-2X_2,\quad [X_3,X_5]=2X_4,\quad [X_4,X_5]=-2X_3, \\
%&[P,X_1]=\alpha X_2, \quad [P,X_2]=-\alpha X_1, \quad [P,X_3]=\beta X_4, \quad [P,X_4]=-\beta X_3, \\
%&[X_1,X_3]=[X_1,X_4]=[X_2,X_3]=[X_2,X_4]=[P,X_5]=0 
%\end{align*}
\begin{align}\label{eqn:[]}
\left.\begin{array}{l}
\tilde T(X_1,X_2)=2X_5, \quad  \tilde T(X_1,X_5)=-2X_2, \quad \tilde T(X_2,X_5)=2 X_1, \\
\tilde T(X_3,X_4)=2X_5, \quad  \tilde T(X_3,X_5)=-2X_4, \quad \tilde T(X_4,X_5)=2 X_3, \\
\tilde T(X_1,X_3)=\tilde T(X_1,X_4)=\tilde T(X_2,X_3)=\tilde T(X_2,X_4)=0. 
\end{array}\right\}
\end{align}
%where $\alpha$, $\beta$, $u$, $v\in\mathbb R$ with $u^2+v^2>0$,  $\alpha\beta\neq0$ and $r=\beta/\alpha$.
With the help of  (\ref{eqn:tilde-R2})--(\ref{eqn:[]}), we obtain  a contradiction after putting 
$X=X_1$, $Y=Z=X_j$,  $j\in\{3,4\}$ in  (\ref{eqn:tilde_R}).
Accordingly, this case can not occur and the proof is completed.

\end{document}